\documentclass{amsart}

\usepackage{color,graphicx,amssymb,latexsym,amsfonts,txfonts,amsmath,amsthm}
\usepackage{pdfsync}
\usepackage{amsmath,amscd}
\usepackage[all,cmtip]{xy}

\usepackage{tikz}
\usetikzlibrary{matrix}

\input epsf

\usepackage{hyperref}
\hypersetup{
    colorlinks=true,       % false: boxed links; true: colored links
    linkcolor=blue,          % color of internal links
    citecolor=blue,        % color of links to bibliography
    filecolor=blue,      % color of file links
    urlcolor=blue           % color of external links
}

\newtheorem{theo}{Theorem}
\newtheorem{coro}{Corollary}
\newtheorem{prop}{Proposition}

\newtheorem{question}{Question}

\theoremstyle{remark}
\newtheorem{rema}{\bf Remark}
\newtheorem{example}{\bf Example}

\begin{document}

\title{Generalized Fermat Riemann surfaces of infinite type}

\author{Rub\'en A. Hidalgo}
\address{Departamento de Matem\'atica y Estad\'{\i}stica, Universidad de La Frontera. Temuco, Chile}
\email{ruben.hidalgo@ufrontera.cl}

\thanks{Partially supported by Projects Fondecyt 1230001 and 1220261}


%%%%%%%%%%%%%%%%%
%%%%%%%%%%%%%%%%%

\begin{abstract}
The  Loch Ness monster (LNM) is, up to homeomorphisms, the unique orientable, connected, Hausdorff, second countable surface of infinite genus and with exactly one end. 
For each integer $k \geq 2$, we construct Riemann surface structures $S$ on the LNM admitting a group of conformal automorphisms $H \cong {\mathbb Z}_{k}^{\mathbb N}$ such that $S/H$ is planar. These structures can be described algebraically inside the projective space ${\mathbb P}^{\mathbb N}$ after deleting some limit points.
\end{abstract}

\maketitle

%%%%%%%%%%%%%%%%%
%%%%%%%%%%%%%%%%%
\section{Introduction}
A closed Riemann surface $S$ is called a generalized Fermat curves (GFC) of type $(k,n)$, where $k,n \geq 2$ are integers, if it 
admits a group $H \cong {\mathbb Z}_{k}^{n}$ of conformal automorphisms (called a generalized Fermat group of type $(k,n)$) such that the quotient Riemann orbifold $S/H$ has genus zero and exactly $n+1$ cone points, each one of order $k$ (in other words, the homology cover of a Riemann orbifold of genus zero with exactly $n+1$ cone points of order $k$). These Riemann surfaces are non-hyperelliptic and they can be algebraically described by  
a smooth and irreducible complex projective algebraic curve (a suitable fiber product of $n-1$ classical Fermat curves of degree $k$) reflecting the action of $H$ \cite{GHL09}.
Moreover, in \cite{HKLP17}, it was proved that the generalized Fermat group is unique under the assumption that $(k-1)(n-1)>2$ (this is similar to the hyperelliptic involution in the case of hyperelliptic Riemann surfaces). This paper aims to extend the above to the case of non-compact Riemann surfaces.

Let $k \geq 2$ be an integer. We say that a non-compact (connected) Riemann surface $S$ is a {\bf ${\mathbb Z}_{k}^{\infty}$-gonal surface} if:
\begin{enumerate}
\item[(i)] $S$ admits a group $H \cong {\mathbb Z}_{k}^{\mathbb N}$ (the infinite countable product of cyclic groups of order $k$) of conformal automorphisms, 
\item[(ii)] there is an infinite discrete set ${\mathcal B} \subset \widehat{\mathbb C}$ such that its limit set ${\mathcal B}' \neq \emptyset$ satisfies that $\widehat{\mathbb C} \setminus {\mathcal B}'$ is connected, and 
\item[(iii)] there is a Galois branched covering $\pi:S \to \widehat{\mathbb C} \setminus {\mathcal B}'$, with deck group $H$, whose branch locus is ${\mathcal B}$  (each branch value of order $k$). 
\end{enumerate}

In the above situation, we also say that $H$ is a {\bf ${\mathbb Z}_{k}^{\infty}$-gonal group} and that $(S,H)$ is a {\bf ${\mathbb Z}_{k}^{\infty}$-gonal pair}.

If, moreover, (a maximality condition)
\begin{enumerate}
\item[(iv)] for every Riemann surface $R$ admitting a group of automorphisms ${\mathbb Z}_{k}^{\mathbb N} \cong G<{\rm Aut}(R)$ such that there is a Galois branched covering map $P:R \to \widehat{\mathbb C} \setminus {\mathcal B}'$, with deck group $G$, whose branch value set is ${\mathcal B}$, and each branch value of order $k$, then there is a subgroup $J<H$ (acting freely on $S$) such that, up to biholomorphisms, $R=S/J$ and $G=H/J$,
\end{enumerate}
then we say that $S$ is a {\bf generalized Fermat curve of type $(k,\infty)$} (or of infinite type if $k$ is clear), 
that $H$ is a {\bf generalized Fermat group of type $(k,\infty)$} and that $(S,H)$ is a {\bf generalized Fermat pair of type $(k,\infty)$}.

Note that, in the case of generalized Fermat pairs of finite type, we didn't need to add the maximality condition (iv) as it is automatically satisfied by Corollary \ref{teo2}. In the infinite type situation, this condition (iv) is not a consequence of the conditions (i)-(iii). 

\medskip

Our first result is Theorem \ref{PropLNM}, which states that every ${\mathbb Z}_{k}^{\infty}$-gonal surface, in particular every generalized Fermat curve of type $(k,\infty)$, is homeomorphic to the Loch Ness monster (LNM), that is, up to homeomorphisms the unique orientable, connected, Hausdorff, second countable surface of infinite genus and with exactly one end. The classification of 
orientable, connected, Hausdorff, and second countable surfaces can be found in \cite{Ker,Ian} (see Section \ref{Sec:Clasifica}).

So, if $X_{L}$ denotes the LNM, then the above provides some classes of Riemann surface structures on it. Other Riemann surface structures on $X_{L}$ where provided 
in \cite{AGHQR, CHTV, HM} in terms of dessins dénfants and origamis. Also, if we consider a hyperbolic Riemann surface $R={\mathbb H}^{2}/\Gamma$ with a negative Euler characteristic (i.e., its fundamental group is non-abelian) and not homeomorphic to a once-punctured closed surface, then its homology cover $\widetilde{R}={\mathbb H}^{2}/\Gamma'$, where $\Gamma'$ denotes the derived (commutator) subgroup of $\Gamma$ (i.e., the normal subgroup generated by all the commutators) is homeomorphic to the LNM \cite{Basmajian-Hidalgo}.
In \cite{Allcock,APV}, it was observed that if $G$ is any countable group, then it can be realized as a group of orientation-preserving homeomorphism on $X_{L}$, acting freely and properly discontinuously, such that the quotient $X_{L}/G$ is homeomorphic to $X_{L}$. This, in particular, asserts that any Riemann surface structure on $X_{L}=X_{L}/G$ lifts to a Riemann surface structure on $X_{L}$ for which $G$ acts as a group of conformal automorphisms.

\medskip

Let $(S,H)$ be a generalized Fermat pair of type $(k,\infty)$ and let 
${\mathcal B}$ and ${\mathcal B}'$ be the associated sets as in the above definition. 
Up to a M\"obius transformation, we may assume ${\mathcal B}=\{\infty,0,1,\lambda_{1}, \lambda_{2},\ldots\} \subset \widehat{\mathbb C}$.
Let ${\mathcal O}$ be the Riemann orbifold whose Riemann surface structure is $\widehat{\mathbb C} \setminus{\mathcal B}'$ and whose set of cone points, each one of order $k$, is ${\mathcal B}$. Let $\Gamma$ be a Fuchsian group such that ${\mathcal O}={\mathbb H}^{2}/\Gamma$. Let $\Gamma_{k}$ be the characteristic subgroup of $\Gamma$ that is generated by its derived subgroup $\Gamma'$ together with the $k$-powers of all of the elements in $\Gamma$. 
Theorem \ref{teouniformiza} asserts that 
$({\mathbb H}^{2}/\Gamma_{k}, \Gamma/\Gamma_{k})$ is a generalized Fermat pair of type $(k,\infty)$ which is isomorphic to $(S,H)$ (i.e., there is biholomorphism $\varphi:{\mathbb H}^{2}/\Gamma_{k} \to S$ such that $\varphi^{-1} H \varphi=\Gamma/\Gamma_{k}$). 
As a consequence, if $(R,K)$ is a ${\mathbb Z}_{k}^{\infty}$-gonal pair, then there is a generalized Fermat pair $(S,H)$ of type $(k,\infty)$ and a subgroup $L \leq H$ such that $R=S/L$ and $K=H/L$.

For each finite set 
$M \subset {\mathcal B}'$ (which it might be empty), we explicitly construct a ${\mathbb Z}_{k}^{\infty}$-gonal surface $C^{k}_{\infty}({\mathcal B};M) \subset {\mathbb P}^{\mathbb N}$ (given as the common zeroes of an infinite set of polynomials, from which we have deleted certain closed subset of a Cantor set) and a ${\mathbb Z}_{k}^{\infty}$-gonal group $H^{k}_{\infty} \leq {\rm Aut}(C^{k}_{\infty}({\mathcal B};M))$, such that there is a Galois branched covering $\pi_{\infty}:C^{k}_{\infty}({\mathcal B};M) \to \widehat{\mathbb C} \setminus {\mathcal B}'$, with deck group $H^{k}_{\infty}$, and whose branch points are exactly the points in ${\mathcal B}$, each one of order $k$ (Theorem \ref{main0}), in particular, that $(C^{k}_{\infty}({\mathcal B};M),H^{k}_{\infty})$ is ${\mathbb Z}_{k}^{\infty}$-gonal pair. In general, 
it might fail the maximality condition, i.e., it might be that $(C^{k}_{\infty}({\mathcal B};M),H^{k}_{\infty})$ is not a generalized Fermat pair of type $(k,\infty)$.
Theorem \ref{main}, asserts that in the case that ${\mathcal B}'$ is finite, then $(S,H)$ is biholomorphically equivalent to $(C^{k}_{\infty}({\mathcal B};M),H^{k}_{\infty})$, where $M$ is obtained from ${\mathcal B}'$ by deleting one of its points. (A fiber product interpretation is also provided in Section \ref{Sec:fiber}).
 
\medskip

Finally, in Section \ref{Sec:hiper}, we consider hyperelliptic Riemann surface structures of infinite type. These are obtained as quotients of generalized Fermat curves of type $(2,\infty)$ (by an index two subgroup of the corresponding generalized Fermat group). Such quotients are homeomorphic to the LNM if and only if $\#{\mathcal B}'=1$ (Proposition \ref{prop3}). In the case that ${\mathcal B}'$ is finite, we also provide algebraic equations of these surfaces (Section \ref{Sec:ecuaciones}).

%%%%%%%%%%%%%
\subsection*{Notations}
If $S$ is a Riemann surface, then we will denote by ${\rm Aut}(S)$ its group of conformal automorphisms. The cyclic group of order $k$ will be denoted by ${\mathbb Z}_{k}$, by ${\mathbb Z}_{k}^{n}$ the direct product of $n$ copies of ${\mathbb Z}_{k}$. If ${\mathbb N}=\{1,2,3,\ldots\}$, the set of natural numbers, then we denote by ${\mathbb Z}_{k}^{\mathbb N}$ the product of countable infinite copies of ${\mathbb Z}_{k}$. For $J$ any set, we denote by $J^{\mathbb N}$ the cartesian product of infinite countable copies of $J$.

%%%%%%%%%%%%%%%%%%%%
%%%%%%%%%%%%%%%%%%
\section{Preliminaries}
%%%%%%%%%%%%%%%

%%%%%%%%%%%%%%%
\subsection{Classification of surfaces}\label{Sec:Clasifica}
We proceed to recall the classification of surfaces due to Ker\'ekjart\'o \cite{Ker}  and Richards \cite{Ian}.
In the next, $X$ will denote a given surface. 

\subsubsection{\bf Subsurfaces and the genus of $X$}
By a \emph{subsurface} of $X$ we mean an embedded surface, which is a closed subset of $X$ and whose boundary consists of a finite number of nonintersecting simple closed curves. Note that a subsurface might or not be compact.
The {\it reduced genus} of a compact subsurface $Y$,  with $q(Y)$ boundary curves and Euler characteristic $\chi(Y)$, is the number $g(Y)=1-\frac{1}{2}(\chi(Y)+q(Y))$. 
The {\it genus} of the surface $X$ is the supremum of the genera of its compact subsurfaces; so it can be a non-negative integer or $\infty$. 

A surface $Z$ is said to be {\it planar} if it has genus zero (in other words, $Z$ is homeomorphic to a domain of the $2$-sphere). 

\subsubsection{\bf The ends space of $X$} 
 We first recall the concepts of ends of $X$ following \cite[1. Kapitel]{Freudenthal}.
 
A sequence $(U_n)_{n\in\mathbb{N}}$, of non-empty connected open subsets of $X$,
is called {\it nested} if the following conditions hold:
\begin{enumerate}
\item for each $n\in\mathbb{N}$, the closure $\overline{U}_{n}$ is a subsurface of $X$,

\item $U_{1}\supset U_{2}\supset\ldots$, and

\item for each compact  $K\subset X$ there is $m\in\mathbb{N}$ such that $K\cap U_m =\emptyset$.
\end{enumerate}

Note that the last condition asserts, in particular, that $\cap_{n\in\mathbb{N}}\overline{U_{n}}=\emptyset$.

Two nested sequences $(U_n)_{n\in\mathbb{N}}$ and $(U'_{n})_{n\in\mathbb{N}}$ are {\it equivalent} if for each $n\in\mathbb{N}$ there exist $j,k\in\mathbb{N}$ such that $U_{n}\supset U'_{j}$, and $U'_{n}\supset U_{k}$. We denote by the symbol $[U_{n}]_{n\in\mathbb{N}}$ the equivalence class of $(U_n)_{n\in\mathbb{N}}$ and we called it an {\it end} of $X$.  The set of ends of $X$ is denoted by ${\rm Ends}(X)$ and it has a structure of a topological space (called the {\it ends space}  of $X$) whose basis is given by the sets
\begin{equation*}
U^{*}:=\{[U_{n}]_{n\in\mathbb{N}}\in{\rm Ends}(X)\hspace{1mm}|\hspace{1mm}U_{j}\subset U\hspace{1mm}\text{for some }j\in\mathbb{N}\},
\end{equation*}
where $U \subset X$ is a non-empty open subset whose boundary $\partial U$ is compact. It is known that ${\rm Ends}(X)$ is Hausdorff, totally disconnected
and compact (so a closed subset of the Cantor set) \cite[Theorem 1.5.]{Ray}.

An end $[U_n]_{n\in\mathbb{N}}$ of $X$ is called {\it planar} if there is $l\in\mathbb{N}$ such that the subsurface $\overline{U}_l\subset X$ is planar. The subset of 
${\rm Ends}(X)$, conformed by all ends of $X$ which are not planar ({\it ends having infinity genus}), is denoted by ${\rm Ends}_{\infty}(X)$. It is known that ${\rm Ends}_{\infty}(X)$ is a closed subset of ${\rm Ends}(X)$ (see \cite[p. 261]{Ian}). The following asserts that $(g,{\rm Ends}_{\infty}(X),{\rm Ends}(X))$, where $g$ is the genus of $X$, is a topological invariant.

\begin{theo}[Ker\'ekjart\'o-Richards classification of surfaces {\cite[\S 7, Erster Abschnitt]{Ker}, \cite[Theorem 1]{Ian}}]
Two surfaces $X_1$ and $X_2$  having the same genus are topologically equivalent if and only if there exists a homeomorphism $f: {\rm Ends}(X_1)\to {\rm Ends}(X_2)$ such that $f( {\rm Ends}_{\infty}(X_1))= {\rm Ends}_{\infty}(X_2)$.
\end{theo}

%%%%%%%%%%%%%%%%%%%%%%%%%%%%%%%
\subsubsection{\bf The Loch Ness monster}\label{d:loch_nesss_monster}
 The {\it Loch Ness monster}\footnote{From the historical point of view as shown in \cite{AR}, this
nomenclature is due to A. Phillips and D. Sullivan \cite{Ph-S}.} (LNM) is the unique, up to homeomorphisms, surface of infinite genus and exactly one end (see Figure \ref{F:LMN1}).
\begin{figure}[h!]
\begin{center}	
\begin{tikzpicture}[baseline=(current bounding box.north)]
\begin{scope}[scale=0.6]
\clip (-6,-1.5) rectangle (8,2);
\draw [line width=1pt] (-2.5,0) to[out=90,in=180] (-1.5,1.5);
\draw [line width=1pt] (-1.5,1.5) to[out=0,in=90] (-0.5,0);
\draw [line width=1pt] (-2,0) to[out=90,in=180] (-1.5,1);
\draw [line width=1pt] (-1.5,1) to[out=0,in=90] (-1,0);
%%%%%%%%%%%%%%%%%%%%%%%%%%%%%%%%%%%%%%%%%%
%%%%%%%%%%%%%%%%%%%%%%%%%%%%%%%%%%%%%%%%%%%%
\draw [dashed, line width=1pt]  (-2.25,0) ellipse (2.5mm and 1mm);
        \draw [dashed, line width=1pt]  (-0.75,0) ellipse (2.5mm and 1mm);
        \draw [dashed, line width=1pt]  (-1.5,1.25) ellipse (1mm and 2.5mm);
%%%%%%%%%%%%%%%
\draw [line width=1pt] (0,0) to[out=90,in=180] (1,1.5);
\draw [line width=1pt] (1,1.5) to[out=0,in=90] (2,0);
\draw [line width=1pt] (0.5,0) to[out=90,in=180] (1,1);
\draw [line width=1pt] (1,1) to[out=0,in=90] (1.5,0);
%%%%%%%%%%%%%%%%%%%%%%
%%%%%%%%%%%
\draw [dashed, line width=1pt]  (0.25,0) ellipse (2.5mm and 1mm);
\draw [dashed, line width=1pt]  (1.75,0) ellipse (2.5mm and 1mm);
\draw [dashed, line width=1pt]  (1,1.25) ellipse (1mm and 2.5mm);
%%%%%%%%%%%%%%%%%%%%%%%%%%%%
%%%%%%%%%%%%
\draw [line width=1pt] (2.5,0) to[out=90,in=180] (3.5,1.5);
\draw [line width=1pt] (3.5,1.5) to[out=0,in=90] (4.5,0);
\draw [line width=1pt] (3,0) to[out=90,in=180] (3.5,1);
\draw [line width=1pt] (3.5,1) to[out=0,in=90] (4,0);
%%%%%%%%%%%%%%%%%%%%%%%%
%%%%%%%%%
\draw [dashed, line width=1pt]  (2.75,0) ellipse (2.5mm and 1mm);
\draw [dashed, line width=1pt]  (4.25,0) ellipse (2.5mm and 1mm);
\draw [dashed, line width=1pt]  (3.5,1.25) ellipse (1mm and 2.5mm);
%%%%%%%%%%%%%%%%%%%%%%%%%%%
%%%%%%%%%%%%%%%%%%%%%%
%%%%%%%%%%%%%%%%%%%%%%%%%%%%%%%%%%%
\draw [line width=1pt](-3,-0.5) -- (5,-0.5);
\draw [line width=1pt](-3,0.3) -- (-2.48,0.3);
\draw [line width=1pt](-2,0.3) -- (-1,0.3);
\draw [line width=1pt](-0.5,0.3) -- (0,0.3);
\draw [line width=1pt](0.5,0.3) -- (1.5,0.3);
\draw [line width=1pt](1.96,0.3) -- (2.5,0.3);
\draw [line width=1pt](3,0.3) -- (4,0.3);
\draw [line width=1pt](4.5,0.3) -- (5,0.3);
\draw [line width=1pt](-3,-0.5) -- (-3,0.3);
\draw [line width=1pt](5,-0.5) -- (5,0.3);
\node at (5.5,0) {$\ldots$};
\node at (-3.5,0) {$\ldots$};
\node at (1.1,-1) {$\vdots$};
\end{scope}
\end{tikzpicture} 
\caption{\emph{Loch Ness monster}}
\label{F:LMN1}
\end{center}	
\end{figure}
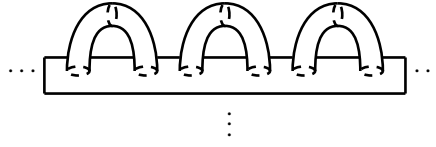

\begin{rema}[{\cite[\S 5.1., p. 320]{SPE}}] The surface $X$ has one end if and only if for all compact subset $K \subset X$ there is a compact $K^{'}\subset X$ such that $K\subset K^{'}$ and $X\setminus  K^{'}$ is connected.
\end{rema}

%%%%%%%%%%%%%%%%
%%%%%%%%%%%%%%%
%\section{Generalized Fermat curves of finite type}
%In this section, we recall the definition of generalized Fermat curves of finite type and some of their properties. 

%%%%%%%%%%%%%%%%
\subsection{The infinite projective space}\label{Sec:projective}
Let ${\mathbb C}^{\mathbb N}$ be the countable product space of the complex plane, with the product space topology (so it is Hausdorff, metrizable and complete, second countable). We consider the infinite projective space (with the quotient topology)
$${\mathbb P}^{\mathbb N}:=({\mathbb C}^{\mathbb N} \setminus \{0\})/{\mathbb C}^{*}.$$

\begin{rema}
In the above notation, ${\mathbb P}^{\mathbb N}$ must not be confused with the infinite projective space ${\mathbb P}^{\infty}$ obtained as the direct limit of the finite-dimensional projective spaces. The direct limit can be identified with the sublocus given by those elements with a finite number of coordinates different from zero.
\end{rema}

We denote by $\pi:{\mathbb C}^{\mathbb N}\setminus\{0\} \to {\mathbb P}^{\mathbb N}$ the quotient map and by the symbol $[x_{1}:x_{2}:\cdots]$ the projection of the point $(x_{1},x_{2},\ldots)$. Note that $\pi$ is continuous (as we are providing to ${\mathbb P}^{\mathbb N}$ with the quotient topology) and open (as, for each $\lambda \in {\mathbb C} \setminus \{0\}$, the bijective map $\phi_{\lambda}(x_{1},x_{2},\ldots)=(\lambda x_{1}, \lambda x_{2}, \ldots)$ is a homeomorphism). The projective space ${\mathbb P}^{\infty}$ is Hausdorff, second countable and it is covered by the open sets $U_{j}=\{[x_{1}:x_{2}:\cdots]: x_{j} \neq 0\}$ (so, it can be seen as an infinite dimensional manifold which is locally homeomorphic to ${\mathbb C}^{\mathbb N}$).

We will need  the following bijective projective linear maps ($\omega_{k}=e^{2 \pi i/k}$):
$$a_{j}([x_{1}:x_{2}:\cdots])=[x_{1}:\cdots:x_{j-1}:\omega_{k} x_{j}: x_{j+1}:\cdots], \quad j=1,2,\ldots.$$

%%%%%%%%%%%%%%%%%%%%%%%%%%
%%%%%%%%%%%%%%%%%%%%%%%%%%%
\subsection{Generalized Fermat curves of finite type}\label{Sec:GFC}
As a consequence of the Riemann-Roch theorem \cite{FK}, every closed Riemann surface $S$ can be described (up to birational isomorphisms) by some smooth irreducible complex algebraic projective curve $C$. In general, if $H$ is a given group of conformal automorphisms of $S$, then finding one of these algebraic curves $C$ from which one can read the action of $H$ is a hard problem. 
For the case that $(S,H)$ is a generalized Fermat pair of type $(k,n)$, this is well known \cite{GHL09} We proceed to recall this below.

Let us consider a generalized Fermat pair $(S,H)$ of type $(k,n)$.

By the definition, there is a Galois branched covering $\pi:S \to \widehat{\mathbb C}$, with deck group $H$, whose branch values set is (up to a M\"obius transformation) 
 $${\mathcal B}_{\pi}=\{\infty, 0, 1, \lambda_{1}, \ldots, \lambda_{n-2}\}.$$

It follows from the Riemann-Hurwitz formula that $S$ has genus
$$g=1+\frac{k^{n-1}\left((k-1)(n-1)-2\right)}{2}.$$

In particular, $S$ is hyperbolic (i.e., its universal cover Riemann surface is the hyperbolic plane ${\mathbb H}^{2}$) if and only if $(k-)(n-1)>2$.

Let us consider the following smooth and irreducible complex projective algebraic curve
$$C^{k}(\lambda_{1},\ldots,\lambda_{n-2}):=
\left\{
\begin{array}{rcc}
x_{1}^{k}+x_{2}^{k}+x_{3}^{k}&=&0\\
\lambda_{1}x_{1}^{k}+x_{2}^{k}+x_{4}^{k}&=&0\\
&\vdots&\\
\lambda_{n-2}x_{1}^{k}+x_{2}^{k}+x_{n+1}^{k}&=&0
\end{array}
\right\}
\subset {\mathbb P}^{n}.
$$

The above curve is invariant under the projective linear maps (where $\omega_{k}=e^{2 \pi i/k}$)
$$a_{j,n}([x_{1}:\cdots:x_{n+1}])=[x_{1}:\cdots:x_{j-1}:\omega_{k} x_{j}: x_{j+1}:\cdots:x_{n+1}], \quad j=1,\ldots,n+1.$$

We may observe that
\begin{enumerate}
\item $a_{1,n}a_{2,n}\cdots a_{n+1,n}=1$;
\item $a_{j,n} \in {\rm Aut}(C^{k}(\lambda_{1},\ldots,\lambda_{n-2}))$ has order $k$;
\item $H_{n}^{k}:=\langle a_{1,n},\ldots,a_{n,n}\rangle \cong {\mathbb Z}_{k}^{n}$;
\item the only non-trivial elements of $H_{n}^{k}$ acting with fixed points on the Riemann surface $C^{k}(\lambda_{1},\ldots,\lambda_{n-2})$ are the powers of $a_{1,n},\ldots,a_{n+1,n}$, and the fixed points of each power $a_{j,n}^{l}$, $l=1,\ldots,k-1$, is also fixed by $a_{j,n}$;
\item the group $H_{n}^{k}$ is the deck group of the Galois branched covering
$$\pi_{n}:C^{k}(\lambda_{1},\ldots,\lambda_{n-2}) \to \widehat{\mathbb C}:[x_{1}:\cdots:x_{n+1}] \mapsto -\left(\frac{x_{2}}{x_{1}}\right)^{k};$$
\item the set of branch values of $\pi_{n}$ is ${\mathcal B}_{\pi}$;
\item $\pi_{n}({\rm Fix}(a_{1,n}))=\infty$, $\pi_{n}({\rm Fix}(a_{2,n}))=0$, $\pi_{n}({\rm Fix}(a_{3,n}))=1$ and, for $j=1,\ldots, n-2$, $\pi_{n}({\rm Fix}(a_{3+j,n}))=\lambda_{j}$,
\end{enumerate}

It follows, from the above, that $(C^{k}(\lambda_{1},\ldots,\lambda_{n-2}),H_{n}^{k})$ is a generalized Fermat pair of type $(k,n)$.

\begin{theo}[\cite{GHL09,HKLP17}]
Within the above notations, the following hold.
\begin{enumerate}
\item There is a biholomorphism $\varphi:S \to C^{k}(\lambda_{1},\ldots,\lambda_{n-2})$ which conjugates $H$ to $H_{n}^{k}$.

\item If $(k-1)(n-1)>2$, then
\begin{enumerate}
\item $S$ is non-hyperelliptic.
\item $H$ is the unique generalized Fermat group of type $(k,n)$ in ${\rm Aut}(S)$.
\item If $\Gamma <{\rm PSL}_{2}({\mathbb R})$ is a Fuchsian group such that ${\mathbb H}^{2}/\Gamma = S/H$ (as orbifolds), then $S={\mathbb H}^{2}/\Gamma'$ and $H=\Gamma/\Gamma'$, where $\Gamma'$ is the commutator subgroup of $\Gamma$.
\end{enumerate}
\end{enumerate}
\end{theo}

\begin{coro}
Let $(S,H)$ be a generalized Fermat pair of type $(k,n)$, where $(k-1)(n-1)>2$, and let $\pi:S \to \widehat{\mathbb C}$ be a Galois branched covering, with deck group $H$. Then there is a short exact sequence
$$1 \to H \to {\rm Aut}(S) \stackrel{\theta}{\to} L \to 1,$$
where (i) $L$ is the group of M\"obius transformations keeping invariant the set of branch values of $\pi$, and (ii) $\theta(\phi) \circ \pi= \pi \circ \phi$, for $\phi \in {\rm Aut}(S)$.
\end{coro}

\begin{coro}\label{teo2}
Let $(S,H)$ be a generalized Fermat pair of type $(k,n)$.
If $R$ is a closed Riemann surface for which there is an abelian group $G<{\rm Aut}(R)$ such that $R/G$ and $S/H$ are biholomorphic as orbifolds, then there is a subgroup $J$ of $H$, acting freely on $S$, and a biholomorphism $\varphi:R \to S/J$ such that $G=\varphi^{-1} (H/J) \varphi$.
\end{coro}
\begin{proof}
Let $\Gamma$ be a Fuchsian group such that $S/H={\mathbb H}^{2}/\Gamma$. Then $S={\mathbb H}^{2}/\Gamma'$ and $H=\Gamma/\Gamma'$. Inside $\Gamma$ there is a torsion free normal subgroup $\Gamma_{R} \lhd \Gamma$ such that, up to biholomorphisms,  $R={\mathbb H}^{2}/\Gamma_{R}$ and $G=\Gamma/\Gamma_{R}$. As $G$ is an abelian group, then $\Gamma' \leq \Gamma_{R}$.
In this case, $J=\Gamma_{R}/\Gamma'$.
\end{proof}

If in the above proof, we allow the normal subgroup $\Gamma_{R}$ to have torsion such that ${\mathbb H}^{2}/\Gamma_{R}$ is an orbifold with all of its cone points of order $k$, then we obtain the following.

\begin{coro}\label{teo3}
Let $(S,H)$ be a generalized Fermat pair of type $(k,n)$.
Let $R$ be a closed Riemann surface and $G<{\rm Aut}(R)$
be an abelian group such that the quotient orbifold $R/G$ has genus zero and all of its cone points have order $k$. Assume there 
a biholomorphism (between the underlying Riemann surfaces) $\varrho:R/G \to S/H$  sending 
the cone points of $R/G$ into the set of cone points of $S/H$. Then there is a subgroup $J$ of $H$ (not necessarily acting freely on $S$) and a biholomorphism (between the underlying Riemann surfaces) $\varphi:R \to S/J$ such that $G=\varphi^{-1} (H/J) \varphi$.
\end{coro}

\begin{example}[Cyclic $p$-gonal curves]
Let $p \geq 2$ be a prime integer and let $R$ be a closed Riemann surface of genus $\gamma \geq 2$ admitting a conformal automorphism $\tau \in {\rm Aut}(R)$ such that $R/\langle \tau \rangle$ has genus zero. In this case, $R$ is called a cyclic $p$-gonal curve, $\tau$ a $p$-gonal automorphism and $\langle \tau \rangle$ a $p$-gonal group of $R$.
We may assume the cone points of $R/\langle \tau \rangle$ is given by a set of the form $\{\infty,0,1,\lambda_{1},\ldots,\lambda_{n-2}\}$, where $2\gamma=(n-1)(p-1)$.  It is known that $R$ can be described by a curve of the form
$$y^{p}=x^{m_{1}}(x-1)^{m_{2}}(x-\lambda_{1})^{m_{3}}\cdots(x-\lambda_{n-2})^{m_{n}},$$
where $m_{j} \in \{1,\ldots,p-1\}$ are such that $m_{1}+\cdots+m_{n}$ is not congruent to zero module $p$. In this model, $\tau(x,y)=(x,\omega_{p}y)$.

Theorem \ref{teo2} asserts the existence of a subgroup $J \cong {\mathbb Z}_{p}^{n-1}$ of $H_{n}^{p}$, acting freely on $C^{p}(\lambda_{1},\ldots,\lambda_{n-2})$ such that 
$R=C^{p}(\lambda_{1},\ldots,\lambda_{n-2})/J$ and $\langle \tau \rangle=H_{n}^{p}/J$.
If $p=2$, so $2\gamma=n-1$, then $J$ is the unique subgroup of $H_{n}^{2}$ which is isomorphic to  ${\mathbb Z}_{2}^{n-1}$ and acting freely on $C^{2}(\lambda_{1},\ldots,\lambda_{n-2})$. This group is given by 
$$J=\langle a_{2,n}a_{1,n}, a_{3,n}a_{1,n},\ldots, a_{n,n}a_{1,n}\rangle$$
and $R$ is the hyperelliptic Riemann surface with the equation
$$y^{2}=x(x-1)(x-\lambda_{1})\cdots(x-\lambda_{n-2}).$$

\end{example}

%%%%%%%%%%%%%%%%%
%%%%%%%%%%%%%%%%%
\section{${\mathbb Z}_{k}^{\infty}$-gonal surfaces are homeomorphic to the LNM}\label{Sec:IGFC}
Let us recall that a Riemann surface $S$ is a ${\mathbb Z}_{k}^{\infty}$-gonal surface if:
\begin{enumerate}
\item[(i)] there exists $H \cong {\mathbb Z}_{k}^{\mathbb N} < {\rm Aut}(S)$; 
\item[(ii)] there is an infinite discrete subset ${\mathcal B}$ of $\widehat{\mathbb C}$ such that, if ${\mathcal B}' \neq \emptyset$ denotes its limit set of points, then $\widehat{\mathbb C} \setminus {\mathcal B}'$ is connected;
\item[(iii)] there is a Galois branched covering map $\pi:S \to \widehat{\mathbb C} \setminus {\mathcal B}'$, with deck group $H$, and whose branch value set is ${\mathcal B}$, each one of order $k$.
\end{enumerate}

\begin{rema}\label{Obs3}
 Up to a M\"obius transformation, we may always assume $\infty,0,1 \in {\mathcal B}$.
\end{rema}

%%%%%%%%%%%%%%%%%%%%%%%
%\subsection{${\mathbb Z}_{k}^{\infty}$-gonal surfaces are homeomorphic to the LNM}

\begin{theo}\label{PropLNM}
If $k \geq 2$, then every ${\mathbb Z}_{k}^{\infty}$-gonal surface, in particular, every generalized Fermat curve of infinite type,  is homeomorphic to the LNM.
\end{theo}
\begin{proof}
Let $(S,H)$ be a ${\mathbb Z}_{k}^{\infty}$-gonal pair, where $k \geq 2$. Let us consider an infinite discrete subset ${\mathcal B}$ of $\widehat{\mathbb C}$ such that $\widehat{\mathbb C} \setminus {\mathcal B}'$ is connected and there is 
a Galois branched covering map $\pi:S \to \widehat{\mathbb C} \setminus {\mathcal B}'$, with deck group $H$, and whose branch value set is ${\mathcal B}$.
Up to post-composing $\pi$ by a suitable M\"obius transformation, we may assume 
${\mathcal B}=\{\infty,0,1,\lambda_{1},\lambda_{2},\ldots\}.$

Let $n \geq 3$ and let $\gamma \subset \widehat{\mathbb C} \setminus \overline{\mathcal B}$
be a simple loop that bounds a topological disc $V_{1}$ containing (in its interior) the points $\infty,0,1,\lambda_{1},\ldots, \lambda_{n-3}$.  Set $V_{2}=\widehat{\mathbb C} \setminus \left(\overline{V_{1}} \cup {\mathcal B}' \right)$.

Let $R_{j}$ be a connected component of $\pi^{-1}(V_{1})$. As $V_{1}$ contains the $\pi$-image of a fixed point of each $a_{1},\ldots,a_{n}$, we observe that $R_{j}$ must contain a fixed point of each $a_{1},\ldots, a_{n}$; so its $H$-stabilizer must contain the subgroup $H_{n,1}=\langle a_{1},\ldots,a_{n}\rangle \cong {\mathbb Z}_{k}^{n}$. Moreover, the loop $\gamma$ lifts under $\pi:R_{j} \to \overline{V_{1}}$ to exactly $k^{n-1}$ pairwise disjoint loops (these being all the boundary loops of $R_{j}$), each one invariant under $\hat{a}_{n+1}=a_{1}\cdots a_{n}$. This permits us to observe that $R_{j}$ is obtained from a generalized Fermat curve $\overline{R}_{j}$ of type $(k,n)$ from which $k^{n-1}$ pairwise disjoint discs have been deleted (each of these discs contains exactly one fixed points of $\hat{a}_{n+1}$) and that its $H$-stabilizer coincides with $H_{n,1}$ (this group being the restriction of the generalized Fermat group of type $(k,n)$ of $\overline{R}_{j}$). 
These connected components $R_{j}$ are permuted by the subgroup $H_{n,2}=\langle a_{n+1},a_{n+1},\ldots\rangle$, and none of them is fixed by a non-trivial element of $H_{n,2}$, so we have a disjoint union $\pi_{\infty}^{-1}(V_{1})=\cup_{j=1}^{\infty} R_{j}$.

Let us observe that, the genus of $R_{j}$ grows to infinity as $n$ tends to infinity. This, in particular, asserts that $S$ has an infinite genus.

Now, let $S_{j}$ be a connected component of $\pi_{\infty}^{-1}(V_{2})$. As $V_{2}$ contains the $\pi$-image of a fixed point of each $a_{j}$, for $j \geq n+1$, we observe that its $H$-stabilizer is the subgroup $H_{n,2}$.  These components are transitively permuted by the finite group $H_{n,1}$. In this way, 
$\pi_{\infty}^{-1}(V_{2})$ is a disjoint union of infinite type connected surfaces $S_{1}$, \ldots, $S_{k^{n-1}}$, each one with an infinite number of boundary loops.

From the above, we have the following facts:
\begin{enumerate}
\item[(1)] each $R_{j}$ shares a boundary loop with each $S_{1},\ldots,S_{k^{n-1}}$, 
\item[(2)] $R_{j}$ and $S_{i}$ have exactly one common boundary loop, and 
\item[(3)] any two surfaces $R_{j_{1}}$ and $R_{j_{2}}$ have no boundary common loops (similarly for each pair $S_{i_{1}}$ and $S_{i_{2}}$).
\end{enumerate}

The gluing pattern of these surfaces can be described by a connected bipartite graph whose black vertices are the surfaces $R_{j}$ and the white ones are the surfaces $S_{1},\ldots, S_{k^{n-1}}$, and there exactly one edge connecting a black vertice with a white vertice. This gluing graph has exactly one end. 

Now, as $S$ has no punctures (since $\widehat{\mathbb C} \setminus {\mathcal B}'$ does not) and the ends of the quotient orbifold $S/H$ are accumulated by the branch values of $\pi_{\infty}$, by making $n$ approaching $\infty$,
the above permits us to see that $S$ has exactly one end.
\end{proof}

%%%%%%%%%%%%%%%%%%%%%%%%
%%%%%%%%%%%%%%%%%%%%%%%%%%
\section{Generalized Fermat pairs of infinite type}\label{Sec:seccion5}
Let us consider a generalized Fermat pair $(S,H)$ of type $(k,\infty)$. By the definition, there is an infinite discrete subset ${\mathcal B} \subset \widehat{\mathbb C}$, such that $\widehat{\mathbb C} \setminus {\mathcal B}'$ is connected and there is a Galois branched covering $\pi:S \to \widehat{\mathbb C} \setminus {\mathcal B}'$, with deck group $H \cong {\mathbb Z}_{k}^{\mathbb N}$, whose set of branch values (each of order $k$) is ${\mathcal B}$. As in the previous section, we may assume 
$${\mathcal B}=\{\infty,0,1,\lambda_{1},\lambda_{2},\ldots\}.$$

%%%%%%%%%%%%%%%%%%%%%%%%
\subsection{Fuchsian description of generalized Fermat pairs}
The following result provides a Fuchsian description of the generalized Fermat pairs.

\begin{theo}\label{teouniformiza}
Let $(S,H)$ be a generalized Fermat pair of type $(k,\infty)$ and let $\Gamma<{\rm PSL}_{2}({\mathbb R})$ be a Fuchsian group such that ${\mathbb H}^{2}/\Gamma=S/H$. 
If $\Gamma'$ denotes the commutator subgroup of $\Gamma$, $\Gamma^{k}$ the subgroup generated by the $k$-powers of the elements of $\Gamma$ and 
$\Gamma_{k}=\langle \Gamma', \Gamma^{k}\rangle$, then
there is a biholomorphism $\varphi:S \to {\mathbb H}^{2}/\Gamma_{k}$ such that $\varphi H \varphi^{-1}=\Gamma/\Gamma_{k}$.
\end{theo}
\begin{proof}
Let us observe that $\Gamma_{k}$ is torsion-free. In this way, $S_{\Gamma}={\mathbb H}^{2}/\Gamma_{k}$ is a Riemann surface admitting $H_{\Gamma}=\Gamma/\Gamma_{k}\cong {\mathbb Z}_{k}^{\mathbb N}$ as a group of conformal automorphisms. Moreover, the inclusion $\Gamma_{k} \lhd \Gamma$ asserts that there is a Galois branched covering $\pi_{\Gamma}:S_{\Gamma} \to \widehat{\mathbb C} \setminus {\mathcal B}'$, with deck group $H_{\Gamma}$, and whose branch values are given by the points in ${\mathcal B}$, each one with branch order $k$.

By covering space theory, there is a torsion-free and normal subgroup $K \lhd \Gamma$ such that $S={\mathbb H}^{2}/K$ and $H=\Gamma/K$.
As $H$ is abelian, $\Gamma' \lhd K$ and, as every non-trivial element of $H$ has order $k$,  that $\Gamma^{k} \lhd K$. All of this asserts that $\Gamma_{k} \leq K$. 
The maximality condition (iv), of our definition, now asserts that $K \leq \Gamma_{k}$.
\end{proof}

\begin{example}\label{ejemplo2}
(1) If $\#{\mathcal B}'=1$, then $\Gamma_{k}=\Gamma'$. 
(2) If $\#{\mathcal B}'=2$, then 
$\Gamma=\langle \alpha,\delta_{1}, \delta_{2},\ldots: \delta_{j}^{k}=1, \; j \geq 1 \rangle$
and
$\Gamma_{k}=\langle \Gamma', \gamma \alpha^{k} \gamma^{-1}: \gamma \in \Gamma\rangle.$
\end{example}

The following asserts, in particular, that every ${\mathbb Z}_{k}^{\infty}$-gonal pair is a quotient of a generalized Fermat pair of type $(k,\infty)$.

\begin{coro}[A universal property of generalized Fermat pairs]\label{universal}
Let $(S,H)$ be a generalized Fermat pair of type $(k,\infty)$ and, as above, let ${\mathcal B} \subset \widehat{\mathbb C}$ be the corresponding branch locus of a Galois branched covering $\pi:S \to \widehat{\mathbb C} \setminus {\mathcal B}'$, with deck group $H \cong {\mathbb Z}_{k}^{\mathbb N}$.
Let $R$ be a Riemann surface and $G<{\rm Aut}(R)$ be an abelian group such that there is a Galois branched covering map $P:R \to \widehat{\mathbb C} \setminus {\mathcal B}'$, with deck group $G$, whose branch value set is ${\mathcal B}$ and each branch value of order $k$. If the finite order elements of $G$ have orders a divisor of $k$,  then there is a normal subgroup $J \lhd H$ (which might contain elements acting with fixed points on $S$) such that, up to biholomorphisms, $R=S/J$ and $G=H/J$.
\end{coro}
\begin{proof}
Let $\Gamma$ be a Fuchsian group such that $S={\mathbb H}^{2}/\Gamma_{k}$ and $H=\Gamma/\Gamma_{k}$.
By the uniformization theorem, there is a torsion-free normal subgroup $F \lhd \Gamma$ such that $R={\mathbb H}^{2}/F$ and $G=\Gamma/F$. As $G$ is abelian, $\Gamma' \lhd F$ and, as every finite order element of $G$ to the power $k$ gives the identity, then $\Gamma^{k} \lhd F$. All of the above asserts that $\Gamma_{k} \lhd F$.
\end{proof}

\begin{question}
Is the generalized Fermat group of type $(k,\infty)$ unique?
\end{question}

%%%%%%%%%%%%
\subsection{An inverse limit construction}
As it was for the finite type situation, we may wonder if there is a way to describe algebraically the pair $(S,H)$ in terms of ${\mathcal B}$.

A first attempt is to consider inverse limits \cite{Spaniel} of the finite type $(k,n)$ generalized Fermat curves $C^{k}(\lambda_{1},\ldots,\lambda_{n-2})$, for $n \geq 3$, and using, for $m \geq n$, the Galois branched coverings 
$$f_{mn}:C^{k}(\lambda_{1},\ldots,\lambda_{m-2}) \to C^{k}(\lambda_{1},\ldots,\lambda_{n-2}): [x_{1}:\cdots: x_{m+1}] \mapsto [x_{1}:\cdots: x_{n+1}].$$ 

Unfortunately, the corresponding inverse limit is not a Riemann surface; the points over which this inverse limit fails to have a Riemann surface structure are those living over the set ${\mathcal B}'$. So, we need to delete these points. 

Below, we proceed to construct a certain type of ``algebraic" model in ${\mathbb P}^{\mathbb N}$ defining Riemann surfaces $C^{k}_{\infty}({\mathcal B},M)$ (homeomorphic to the LNM), where $M \subset {\mathcal B}'$ is finite (which it might or not be empty),  admitting a group of conformal automorphisms $H^{k}_{\infty} \cong {\mathbb Z}_{k}^{\mathbb N}$, such that there is a Galois branched covering $\pi_{\infty}:C^{k}_{\infty}({\mathcal B},M) \to \widehat{\mathbb C} \setminus {\mathcal B}'$ with deck group $H^{k}_{\infty}$ and whose branch values (each one of order $k$) are given by the points in ${\mathcal B}$.

\subsubsection{}
Let us fix a finite subset $M \subset {\mathcal B}'$ (which might or not be empty). If $M \neq \emptyset$, we set $M=\{p_{1},\ldots,p_{N}\}$.

\subsubsection{}
For each $n \geq 3$, we consider the generalized Fermat curve of type $(k,N+n)$
$$C^{k}(p_{1},\ldots,p_{N},\lambda_{1},\ldots,\lambda_{n-2}),$$  its corresponding generalized Fermat group $H^{k}_{n+N} \cong {\mathbb Z}_{k}^{n+M}$, and the Galois branched covering
$$\pi_{n}:C^{k}(p_{1},\ldots,p_{N},\lambda_{1},\ldots,\lambda_{n-2}) \to \widehat{\mathbb C}: [x_{1}:\cdots:x_{n+N+1}] \mapsto -\left(\frac{x_{2}}{x_{1}}\right)^{k},$$
whose deck group is $H^{k}_{n+N}$ and branch values set is  
$$\{\infty,0,1,p_{1},\ldots,p_{N},\lambda_{1},\ldots,\lambda_{n-2}\}.$$

\subsubsection{}
We set 
$$C_{n}(M):=C^{k}(p_{1},\ldots,p_{N},\lambda_{1},\ldots,\lambda_{n-2}) \setminus \pi_{n}^{-1}({\mathcal B}') \subset {\mathbb P}^{n},$$
which is a connected (and non-compact) Riemann surface (which is analytically finite if ${\mathcal B}'$ is finite).
If $M=\emptyset$, then $C_{n}(\emptyset):=C^{k}(\lambda_{1},\ldots,\lambda_{n-2}) \setminus \pi_{n}^{-1}({\mathcal B}') \subset {\mathbb P}^{n}.$

Note that $H^{k}_{n+N}<{\rm Aut}(C_{n}(M))$ is the deck group of the 
Galois branched covering
$$\pi_{n}:C_{n}(M) \to \widehat{\mathbb C}\setminus{\mathcal B}': [x_{1}:\cdots:x_{n+N+1}] \mapsto -\left(\frac{x_{2}}{x_{1}}\right)^{k},$$
whose branch values set is
$$\{\infty,0,1,\lambda_{1},\ldots,\lambda_{n-2}\}.$$

\subsubsection{}
Now, for $m \geq n$, we let
$$f_{mn}:C_{m}(M) \to C_{n}(M): [x_{1}:\cdots: x_{m+N+1}] \mapsto [x_{1}:\cdots: x_{n+N+1}],$$
which is a Galois branched covering whose deck group is given by the subgroup of $H^{k}_{m+N}$ given by $\langle a_{n+N+2,m+N},\ldots,a_{m+N+1,m+N}\rangle \cong {\mathbb Z}_{k}^{m-n}$.
Note that $\pi_{n} \circ f_{mn}=\pi_{m}$.

\subsubsection{}
With the above data, we may consider the inverse limit
$$C^{k}_{\infty}({\mathcal B};M)=\underleftarrow{\lim}(C_{n}(M),f_{mn})=$$
$$=\{p=(p_{n}) \in \prod_{n\geq 3} C_{n}(M): f_{mn}(p_{m})=p_{n}, \; \forall m \geq n \geq 3\},$$
together the map
$$\pi_{\infty}:C^{k}_{\infty}({\mathcal B};M) \to \widehat{\mathbb C}\setminus{\mathcal B}':(p_{n}) \mapsto \pi_{3}(p_{3}).$$

The above inverse limit can be algebraically described as 
$$
C^{k}_{\infty}({\mathcal B};M):=\left\{
\begin{array}{rcl}
x_{1}^{k}+x_{2}^{k}+x_{3}^{k}&=&0\\
p_{1}x_{1}^{k}+x_{2}^{k}+x_{4}^{k}&=&0\\
&\vdots&\\
p_{N}x_{1}^{k}+x_{2}^{k}+x_{3+N}^{k}&=&0\\
\lambda_{1}x_{1}^{k}+x_{2}^{k}+x_{4+N}^{k}&=&0\\
\lambda_{2}x_{1}^{k}+x_{2}^{k}+x_{5+N}^{k}&=&0\\
&\vdots&\\
-\left(\dfrac{x_{2}}{x_{1}}\right)^{k} &\notin &{\mathcal B}'.
\end{array}
\right\} \subset {\mathbb P}^{\mathbb N},
$$
and 
$$\pi_{\infty}:C^{k}_{\infty}({\mathcal B};M) \to \widehat{\mathbb C} \setminus {\mathcal B}': [x_{1}:x_{2}:\cdots] \mapsto -\left(\dfrac{x_{2}}{x_{1}}\right)^{k}.$$

%%%%%%%%%%%%%%%%%%%
\subsubsection{}\label{Sec:FP}
The subset $C^{k}_{\infty}({\mathcal B};M)$ is invariant under each of the maps $a_{j}$ (as described in Section \ref{Sec:projective}). So $C^{k}_{\infty}({\mathcal B};M)$ admits  
$$H_{\infty}^{k}=\langle a_{1},a_{2},\ldots \rangle \cong {\mathbb Z}_{k}^{\mathbb N}$$
as a group of bijections and 
$$ \pi_{\infty} \circ a_{j}=\pi_{\infty}, \quad j=1,2,\ldots$$

As the form of the elements $a_{j}$ are diagonals, it can be seen the following ($\rho_{k}:=e^{\pi i/k}$).
\begin{enumerate}
\item ${\rm Fix}(a_{1})=\{x_{1}=0\} \cap C^{k}_{\infty}({\mathcal B};M)=$
$$=\{[0:1:\rho_{k}^{l_{3}}:\rho_{k}^{l_{4}}:\cdots]: l_{j} \in \{1,\ldots,k-1\}\}.$$
\item ${\rm Fix}(a_{2})=\{x_{2}=0\} \cap C^{k}_{\infty}({\mathcal B};M)=$
$$=\{[1:0:\rho_{k}^{l_{3}}: \rho_{k}^{l_{3}}p_{1}^{1/k}:\cdots: \rho_{k}^{l_{3+N}}p_{N}^{1/k} :\rho_{k}^{l_{4+N}}\lambda_{1}^{1/k}:\cdots]: l_{j} \in \{1,\ldots,k-1\}\}.$$

\item ${\rm Fix}(a_{3})=\{x_{3}=0\} \cap C^{k}_{\infty}({\mathcal B};M)=$
$$=\{[1:\rho_{k}^{l_{2}}:0: \rho_{k}^{l_{4}}(1+p_{1})^{1/k}:\cdots: \rho_{k}^{l_{3+N}}(1+p_{N})^{1/k} :\rho_{k}^{l_{4+N}}(1+\lambda_{1})^{1/k}:$$
$$\cdots]: l_{j} \in \{1,\ldots,k-1\}\}.$$

\item ${\rm Fix}(a_{4})=\cdots={\rm Fix}(a_{N+3})=\emptyset$.

\item If $j \geq 4$, ${\rm Fix}(a_{N+j})=\{x_{N+j}=0\} \cap C^{k}_{\infty}({\mathcal B};M)=$
$$=\{[1:\rho_{k}^{l_{2}}\lambda_{j}^{1/k}:\omega_{k}^{l_{3}}(\lambda_{j}-1)^{1/k}: \omega_{k}^{l_{4}}(\lambda_{j}-p_{1})^{1/k}:\cdots: \omega_{k}^{l_{3+N}}(\lambda_{j}-p_{N})^{1/k}:$$
$$\omega_{k}^{l_{4+N}}(\lambda_{j}-\lambda_{1})^{1/k}:\cdots:\omega_{k}^{l_{2+j+N}}(\lambda_{j}-\lambda_{j-1})^{1/k}:0:\omega_{k}^{l_{4+j+N}}(\lambda_{j}-\lambda_{j+1})^{1/k}:$$
$$\cdots]: l_{j} \in \{1,\ldots,k-1\}\}.$$

\item If $b \in H_{\infty}^{k} \setminus \{a_{1}^{l_{1}}, a_{2}^{l_{2}}, a_{3}^{l_{3}}, a_{N+4}^{l_{4}},a_{N+5}^{l_{5}},\ldots: 0 \leq l_{j} \leq k-1\}$, then ${\rm Fix}(b)=\emptyset$.

\end{enumerate}

%%%%%%%%%%%%%%%%%%%
\subsubsection{\bf $C^{k}_{\infty}({\mathcal B};M)$ are ${\mathbb Z}_{k}^{\infty}$-gonal surfaces}

\begin{theo}\label{main0}
Within the above notations, the following hold.
\begin{enumerate}
\item $C^{k}_{\infty}({\mathcal B};M)$ is a connected Riemann surface, homeomorphic to the LNM.

\item ${\mathbb Z}_{k}^{\mathbb N} \cong H_{\infty}^{k}<{\rm Aut}(C^{k}_{\infty}({\mathcal B};M))$.

\item The map $\pi_{\infty}:C^{k}_{\infty}({\mathcal B},M) \to \widehat{\mathbb C} \setminus {\mathcal B}'$ is a Galois branched covering with deck group $H_{\infty}^{k}$, whose branch locus set is ${\mathcal B}$.

\item The only elements of $H_{\infty}^{k}$ acting with fixed points on $C^{k}_{\infty}({\mathcal B};M)$ are the powers of the elements $a_{j}$, where $j=1,2,3,4+N,5+N, \ldots$.

\end{enumerate}
\end{theo}
\begin{proof}
Once we have proved that $C^{k}_{\infty}({\mathcal B};M)$ is a Riemann surface, we may observe that $H_{\infty}^{k} \cong {\mathbb Z}_{k}^{\mathbb N}$ is a group of conformal automorphisms and that it is the deck group of the holomorphic branched covering $\pi_{\infty}$ (providing parts (2) and (3)). All the above will ensure that $(C^{k}_{\infty}({\mathcal B};M),H^{k}_{\infty})$ is a ${\mathbb Z}_{k}^{\infty}$-gonal pair, in particular, (by Theorem \ref{PropLNM} ) that $C^{k}_{\infty}({\mathcal B};M)$ is homeomorphic to the LNM.
Part (4) is a consequence of Section \ref{Sec:FP}.

\subsubsection*{(i) $C^{k}_{\infty}({\mathcal B};M)$ is a Riemann surface}

Let $p=[p_{1}:p_{2}:\cdots] \in  C^{k}_{\infty}({\mathcal B};M)$, so $u_{0}=-(p_{2}/p_{1})^{k} \notin {\mathcal B}'$.

Let us first assume $u_{0} \notin {\mathcal B}$. As ${\mathcal B}$ is a discrete set and $u_{0} \notin \overline{\mathcal B}={\mathcal B} \cup {\mathcal B}'$, we may find an small disc $D(u_{0};r) \subset \widehat{\mathbb C} \setminus \overline{\mathcal B}$. Inside such a disc, we may find local branches of $k$-roots: $\sqrt[k]{u}$, $\sqrt[k]{u-1}$ and $\sqrt[k]{u-\lambda_{j}}$, for $j \geq 1$. This permits to define a local chart
$$u \in D(u_{0};r) \mapsto [1:\sqrt[k]{u}:\sqrt[k]{u-1}:\sqrt[k]{u-\lambda_{1}}:\cdots] \in C^{k}_{\infty}({\mathcal B};M).$$

The above permits to see that $C^{k}_{\infty}({\mathcal B};M) \setminus \pi_{\infty}^{-1}({\mathcal B})$ has the structure of a Riemann surface.

Next, as $\pi_{\infty}^{-1}({\mathcal B})$ is a discrete subset of $C^{k}_{\infty}({\mathcal B};M)$ and $\pi_{\infty}$ has finite local degree at each of these points (degree $k$), we may observe that the above Riemann surface structure extends to a Riemann surface structure on all $C^{k}_{\infty}({\mathcal B};M)$ making $\pi_{\infty}$ a holomorphic branched covering onto $\widehat{\mathbb C} \setminus {\mathcal B}'$, whose branch values are the elements of ${\mathcal B}$, each one of branch order $k$.

\subsubsection*{(ii) $C^{k}_{\infty}({\mathcal B};M)$ is connected}
To check the connectivity, we consider a simple path $\delta \subset \widehat{\mathbb C} \setminus {\mathcal B}'$, which starts at $\infty$ and connects to $0$, then to $1$, then to $\lambda_{1}$, then to $\lambda_{2}$, etc. By lifting $\delta$ under $\pi_{\infty}$, we obtain an infinite connected graph ${\mathcal G} \subset C^{k}_{\infty}({\mathcal B})$, whose vertices are all of the fixed points of the elements $a_{1}, a_{2}, \ldots$, and each one has degree $k$. Each connected component of $C^{k}_{\infty}({\mathcal B};M) \setminus {\mathcal G}$ is a topological disc (homeomorphic, under $\pi_{\infty}$, to $\widehat{\mathbb C} \setminus \left({\mathcal B}' \cup \delta\right)$ under $\pi_{\infty}$).
\end{proof}

%%%%%%%%%%%%%%%%%%

\begin{rema}[An unbranched quotient]\label{Sec:cociente}
Let us assume $M \neq \emptyset$. In this case, the Riemann surface $C^{k}_{\infty}({\mathcal B};M)$ admits the group $H^{k}_{\infty}=\langle a_{1},a_{2}.\ldots\rangle \cong {\mathbb Z}_{k}^{\mathbb N}$ as a group of conformal automorphisms. The subgroup $L=\langle a_{4},\ldots,a_{N+3}\rangle \cong {\mathbb Z}_{k}^{N}$ acts freely on $C^{k}_{\infty}({\mathcal B};M)$ and it is the deck group of the 
Galois (unbranched) covering
$$\pi_{L}:C^{k}_{\infty}({\mathcal B};M) \to C^{k}_{\infty}({\mathcal B};\emptyset)$$
$$[x_{1}:x_{2}:\cdots] \mapsto [x_{1}:x_{2}:x_{3}:x_{N+4}:x_{N+5}: \cdots].$$
\end{rema}

%%%%%%%%%%%%%%%%%%%
\subsection{\bf Algebraic equations for $(S,H)$ if ${\mathcal B}'$ finite}
In the above construction, we have obtained Riemann surface structures on the LNM (which are algebraically defined) and which satisfy the first three conditions of a generalized Fermat curve of type $(k,\infty)$. But, it might be that the maximality condition does not hold. Below, we restrict ourselves to the case when ${\mathcal B}'$ is finite and we show that, by taking $M$ equal to ${\mathcal B}'$ minus a point, we obtain a generalized Fermat curve of type $(k,\infty)$.

Let ${\mathcal B}'=\{p_{1},\ldots, p_{N+1}\}$, where $N \geq 0$, and set $M=\{p_{1},\ldots,p_{N}\}$ (if $N=0$, $M=\emptyset$).

\begin{theo}\label{main}
Under the above finiteness assumption,  
$(C^{k}_{\infty}({\mathcal B};M),H_{\infty}^{k})$ is a generalized Fermat pair of type $(k,\infty)$ and 
there a biholomorphism between $(S,H)$ and $(C^{k}_{\infty}({\mathcal B};M),H_{\infty}^{k})$.
\end{theo}
\begin{proof}
Let $\Gamma$ be a Fuchsian group such that ${\mathbb H}^{2}/\Gamma=S/H$. Then $S={\mathbb H}^{2}/\Gamma_{k}$ and $H=\Gamma/\Gamma_{k}$.
In this case, we have that $$\Gamma=\langle \alpha_{1},\ldots,\alpha_{N}, \delta_{1},\delta_{2},\ldots: \delta_{j}^{k}=1, \; j=1,2,\ldots\rangle. $$

By the maximality condition (in the definition of generalized Fermat curves), there is a torsion-free normal subgroup $F$ of $\Gamma$ such that $C^{k}_{\infty}({\mathcal B};M)={\mathbb H}^{2}/F$ and $H^{k}_{\infty}=\Gamma/F$. As $H^{k}_{\infty} \cong {\mathbb Z}_{k}^{\mathbb N}$, $\Gamma_{k} \lhd F$. In this way, $J=F/\Gamma_{k} \lhd H$ is such that $S/J=C^{k}_{\infty}({\mathcal B};M)$ and $H/J=H^{k}_{\infty}$. We need to prove that $J$ is the trivial group. 

Let us assume there is some $b \in J$, $b \neq 1$. Then, we may write $b=t_{i_{1}}\cdots t_{i_{r}} d_{j_{1}}d_{j_{2}}\cdots d_{j_{m}}$, where $d_{j}$ (respectively, $t_{i}$)  is the $\Gamma_{k}$-class of $\delta_{j}$ (repspectively, $\alpha_{i}$). 
The element $b$ induces the identity element in $H^{k}_{\infty}$. Each $t_{i_{l}}$ induces a generator $a_{i_{l}} \in \{a_{4},\ldots,a_{N+3}\}$ (each of them has not fixed points) and
each $d_{j_{l}}$ induces a generator $a_{j_{l}} \in \{a_{1},a_{2},a_{3},a_{N+4},a_{N+5},\ldots\}$ (each of them has fixed points).
It follows that $1=a_{i_{1}} \cdots a_{i_{r}} a_{j_{1}} a_{j_{2}}\cdots a_{j_{m}}$.

By the (diagonal) form of the elements $a_{j}$, we may observe that this obligates to have $r=0$, and that $j_{1}=\cdots=j_{m}=i$. that is, $b \in \langle \delta_{i} \rangle$, for some fixed $i$. As $b$ cannot have fixed points, then $b=1$, a contradiction to our assumption.
\end{proof}

%%%%%%%%%%%%%%%%%
\subsubsection{\bf A fiber product interpretation}\label{Sec:fiber}
Let us assume, as above, that $\#{\mathcal B}'=1+N$, and 
that $N\geq 1$. In this case, 
$\Gamma=\langle \alpha_{1},\ldots,\alpha_{N},\delta_{1},\delta_{2},\ldots: \delta_{j}^{k}=1, \; j \geq 1\rangle$ and 
$$H=\Gamma/\Gamma_{k}=H_{1} \times H_{2}$$
where 
$$H_{1}=\langle \alpha_{1},\ldots, \alpha_{N}\rangle/\Gamma_{k}=\langle b_{1},\ldots,b_{N}\rangle \cong {\mathbb Z}_{k}^{N},$$
$$H_{2}=\langle \delta_{1}, \delta_{2},\ldots \rangle/\Gamma_{k}= \langle d_{1}, d_{2}, \ldots \rangle \cong {\mathbb Z}_{k}^{\mathbb N}.$$

The group $H_{1}$ acts freely on $S={\mathbb H}/\Gamma_{k}$ and every $d_{j}$ acts with fixed points (in $H_{2}$, the only elements acting with fixed points are the powers of these generators).

The quotient Riemann surface $S/H_{1}$ is uniformized by the Fuchsian group $\langle \Gamma', \gamma \alpha_{j} \gamma^{-1}: j=1,\ldots,N, \gamma \in \Gamma\rangle \lhd \Gamma$ and it is isomorphic to 
$C^{k}_{\infty}({\mathcal B};\emptyset)$. Also, $H^{k}_{\infty}=H/H_{1}$. Moreover, if $P_{1}:S \to C^{k}_{\infty}({\mathcal B};\emptyset)$ is a Galois (unbranched) covering with deck group $H_{1}$, then the Galois branched covering $\pi:S \to \widehat{\mathbb C} \setminus {\mathcal B}'$, with $H$ as its deck group, is given by $\pi_{\infty} \circ P_{1}$.

Note that, as seen in Remark \ref{Sec:cociente}, we also have the Galois (unbranched) covering $\pi_{L}:C^{k}_{\infty}({\mathcal B};M) \to C^{k}_{\infty}({\mathcal B};\emptyset)$, where 
$L=\langle a_{4},\ldots,a_{N+3}\rangle \cong {\mathbb Z}_{k}^{N}$ acts freely on $C^{k}_{\infty}({\mathcal B}:M)$.

The Riemann surface structure $R^{*}$ of the quotient orbifold $S/H_{2}$ is given as the complement of $Nk^{N-1}$ points (these are the preimages of the set ${\mathcal B}'$) of the 
generalized Fermat curve of type $(k,N)$
$$R:=
\left\{
\begin{array}{rcc}
y_{1}^{k}+y_{2}^{k}+y_{3}^{k}&=&0\\
\mu_{1}y_{1}^{k}+y_{2}^{k}+y_{4}^{k}&=&0\\
&\vdots&\\
\mu_{N-2}y_{1}^{k}+y_{2}^{k}+y_{N+1}^{k}&=&0
\end{array}
\right\},
$$
where $$\mu_{j}=\frac{(p_{3}-p_{1})(p_{3+j}-p_{2})}{(p_{3}-p_{2})(p_{3+j}-p_{1})},\; j=1,\ldots,N-2,$$
and whose generalized Fermat group is $H_{N}^{k}=H/H_{2}$. The map
$$\pi_{1}:R \to \widehat{\mathbb C}:[y_{1}:\cdots:y_{N+1}] \mapsto \frac{p_{2}(p_{3}-p_{1})y_{1}^{k}+p_{1}(p_{3}-p_{2})y_{2}^{k}}{(p_{3}-p_{1})y_{1}^{k}+(p_{3}-p_{2})y_{2}^{k}}$$
is a Galois branched covering with deck group $H_{N}^{k}$ and whose branch values are $p_{1},\ldots,p_{N+1}$. Note that $\pi_{1}=T\circ \pi_{N}$, where $T$ is the M\"obius transformation such that $T(\infty)=p_{1}$, $T(0)=p_{2}$ and $T(1)=p_{3})$.

The above permits to observe that $S$ is a connected components of the fiber product of $(R^{*},\pi_{1})$ with $(C^{k}_{\infty}({\mathcal B};\emptyset),\pi_{\infty})$. Such a fiber product is given by those pairs $$([y_{1}: \cdots:y_{N+1}],[x_{1}:x_{2}:\cdots]) \in {\mathbb P}^{N} \times {\mathbb P}^{\mathbb N}$$
$$
\left\{
\begin{array}{rcl}
y_{1}^{k}+y_{2}^{k}+y_{3}^{k}&=&0\\
\mu_{1}y_{1}^{k}+y_{2}^{k}+y_{4}^{k}&=&0\\
&\vdots&\\
\mu_{N-2}y_{1}^{k}+y_{2}^{k}+y_{N+1}^{k}&=&0\\
x_{1}^{k}+x_{2}^{k}+x_{3}^{k}&=&0\\
\lambda_{1}x_{1}^{k}+x_{2}^{k}+x_{4}^{k}&=&0\\
\lambda_{2}x_{1}^{k}+x_{2}^{k}+x_{5}^{k}&=&0\\
&\vdots&\\
-\left(\dfrac{x_{2}}{x_{1}}\right)^{k} &\notin &{\mathcal B}'\\
\dfrac{p_{2}(p_{3}-p_{1})y_{1}^{k}+p_{1}(p_{3}-p_{2})y_{2}^{k}}{(p_{3}-p_{1})y_{1}^{k}+(p_{3}-p_{2})y_{2}^{k}}&=&-\left(\dfrac{x_{2}}{x_{1}}\right)^{k}
\end{array}
\right\}.
$$

We also have the projection 
$$\pi_{*}([y_{1}:\cdots:y_{N+1}), [x_{1}:x_{2}:\cdots])= -\left(\dfrac{x_{2}}{x_{1}}\right)^{k}$$
restrict to the connected component $S$ to provide the Galois branched covering $\pi:S \to \widehat{\mathbb C} \setminus {\mathcal B}'$.

We conjecture that all the connected components of the above fiber product are biholomorphic between them (so they are biholomorphic to $C^{k}_{\infty}({\mathcal B};M)$).

\begin{rema}
For $N=1$, in the above we take $R^{*}={\mathbb C} \setminus \{0\}$ and 
$\pi_{1}(z)=(p_{1}-z^{2}p_{2})/(1-z^{2})$.

\end{rema}

%%%%%%%%%%%%%%%%%%%%%%%
%%%%%%%%%%%%%%%%%%%%%%%
\section{Hyperelliptic Riemann surfaces of infinite type}\label{Sec:hiper}
A closed Riemann surface $R$ of genus $g \geq 2$ is called hyperelliptic if it admits a conformal automorphism $\tau$ of order two having exactly $2(g+1)$ fixed points (equivalently, there is a degree two branched covering $\pi:R \to \widehat{\mathbb C}$). In this case, $\tau$ is unique and it is called the hyperelliptic involution of $R$.
These hyperelliptic surfaces are unbranched quotients of generalized Fermat curves of type $(2,2g+1)$ by the unique index two subgroup of its corresponding generalized Fermat group which acts freely.

In this section, we consider hyperelliptic Riemann surfaces of infinite type.

\subsection{Hyperelliptic Riemann surfaces of infinite type}
A non-compact Riemann surface $R$ is called a hyperelliptic Riemann surface of infinite type if:
\begin{enumerate}
\item there exists an infinite discrete subset ${\mathcal B} \subset \widehat{\mathbb C}$ such that, if ${\mathcal B}'$ denotes its limit set, then $\widehat{\mathbb C} \setminus {\mathcal B}'$ is connected; and
\item there exists a degree two branched covering $\pi:R \to \widehat{\mathbb C}\setminus {\mathcal B}'$ whose branch locus set is ${\mathcal B}$. In this case,  the deck group of $\pi$ is $\langle \tau \rangle$, where $\tau \in {\rm Aut}(R)$ has order two; we say that $\tau$ is a hyperelliptic involution of $R$. 
\end{enumerate}

\begin{prop}\label{prop3}
Let $R$ be a hyperelliptic Riemann surface of infinite type and $\tau \in {\rm Aut}(R)$ be a hyperelliptic involution. Then $R$ has an infinite genus and every end is fixed by $\tau$ (so, there is a bijection between the ends of $R$ and the ends of $\widehat{\mathbb C} \setminus {\mathcal B}'$). In particular, $\#{\mathcal B}'=1$ if and only if $R$ is homeomorphic to the LNM.
\end{prop}
\begin{proof}
Let $R$ be a hyperelliptic Riemann surface of infinite type, $\tau$ a hyperelliptic involution and let $\pi:R \to \widehat{\mathbb C} \setminus {\mathcal B}'$ be a two-fold branched covering with branch locus ${\mathcal B}$ (as in the above definition) and deck group $\langle \tau \rangle$.
We may assume ${\mathcal B}=\{\infty,0,1,\lambda_{1},\lambda_{2},\ldots\}$. 
Let $\gamma_{n}$ be a simple closed loop surrounding a disc $D_{n}$ that contains the set $\{\infty,0,1,\lambda_{1},\ldots,\lambda_{2n}\}$ and such that all other points in ${\mathcal B} \cup {\mathcal B}'$ are in the complementary region. The lifting of $\gamma_{n}$ under $\pi$ is a simple loop in $R$, invariant under the corresponding hyperelliptic involution $\tau$. The region $R_{n}:=\pi^{-1}(\overline{D_{n}}) \subset R$ is homeomorphic to a closed surface of genus $g_{n}=n+1$, which is also invariant under $\tau$ (one may think of $R_{n}$ as a hyperelliptic Riemann surface from which we have deleted an open disc neighborhood of a fixed point of its hyperelliptic involution). In this way, we may observe that $R$ has an infinite genus. 
Next, as each end of the quotient orbifold is accumulated by branch values, it follows that each end of $R$ can be approached by fixed points of $\tau$. As the action of $\tau$ extends continuously to an action of the ends, we may observe that each end must be fixed by that involution.

\end{proof}

\begin{rema}
For any two pairs $(R_{1},\tau_{1})$ and $(R_{2},\tau_{2})$, where $R_{j}$ is hyperelliptic and $\tau_{j} \in {\rm Aut}(R_{j})$ is a hyperelliptic involutions, there is a homeomorphism $\varphi:R_{1} \to R_{2}$ such that $\varphi \circ \tau_{1} \circ \varphi^{-1}=\tau_{2}$ if and only if the quotient orbifolds $R_{1}/\tau_{1}$ and $R_{2}/\tau_{2}$ are topologically equivalent as orbifolds.
\end{rema}

%%%%%%%%%
\subsection{Uniformization and connection to generalized Fermat curves of type $(2,\infty)$}
Let $R$ be a hyperelliptic Riemann surface of infinite type, $\tau$ a hyperelliptic involution and let $\pi:R \to \widehat{\mathbb C} \setminus {\mathcal B}'$ be a two-fold branched covering with branch locus ${\mathcal B}$ (as in the above definition) and deck group $\langle \tau \rangle$.
Let us fix a Fuchsian group $\Gamma$ such that $R/\langle \tau \rangle = {\mathbb H}^{2}/\Gamma$.

There is an index two torsion free subgroup $\Gamma_{R} \lhd \Gamma$ such that $R={\mathbb H}^{2}/\Gamma_{R}$ and $\langle \tau \rangle =\Gamma/\Gamma_{R}$.

\begin{rema}
A difference with the case of closed hyperelliptic Riemann surfaces, in which there is a unique index two torsion-free subgroup of the corresponding group $\Gamma$, is that, for $\#{\mathcal B}>1$, the group $\Gamma$ has many different index two torsion-free subgroups.
\end{rema}

Associated to ${\mathcal B}$ is the generalized Fermat pair $(S={\mathbb H}^{2}/\Gamma_{2},H=\Gamma/\Gamma_{2})$ of type $(2,\infty)$, where
$\Gamma_{2}=\langle \Gamma',\Gamma^{2}\rangle=\Gamma^{2}$, 
such that $S/H=R/\langle \tau \rangle$. By the maximality condition of the generalized Fermat pair $(S,H)$, there is an index two subgroup $H_{R}$ of $H$, acting freely on $S$, such that $R=S/H_{R}$ and $\langle \tau \rangle = H/H_{R}$. In this case, $H_{R}=\Gamma_{R}/\Gamma_{2}$. Below, we make explicit the above in the case when ${\mathcal B}'$ is finite.

%%%%%%%%%%%%%%%
\subsection{Hyperelliptic Riemann surfaces for ${\mathcal B}'$ finite}
Let $R$ be a hyperelliptic Riemann surface of infinite type and $\tau \in {\rm Aut}(R)$ be a hyperelliptic involution of $R$. Let us assume, following the previous notations, that
${\mathcal B}'$ is finite, say 
$${\mathcal B}'=\{q_{1},\ldots,q_{N+1}\} \; (N \geq 0).$$

We may assume 
$${\mathcal B}=\{\infty,0,1,\lambda_{1},\lambda_{2},\ldots\}=B_{1} \sqcup \cdots \sqcup B_{N+1},$$
where each $B_{j}$ has infinite cardinality and its unique limit point is $q_{j}$. Note that the partition $\{B_{1},\ldots,B_{N+1}\}$ is not unique.

%%%%%%%%%%%%%
\subsubsection{\bf Fuchsian groups description}
In this case, $R/\langle \tau \rangle = {\mathbb H}^{2}/\Gamma$, where
$$\Gamma=\langle \alpha_{1},\ldots, \alpha_{N},\delta_{1}, \delta_{2}, \ldots : \delta_{j}^{2}=1\rangle.$$

As previously noted, there is a torsion-free index two subgroup $\Gamma_{R}$ of $\Gamma$ such that $R={\mathbb H}^{2}/\Gamma_{R}$ and $\langle \tau \rangle = \Gamma/\Gamma_{R}$.

The collection of all such subgroups is given by the kernels $\Gamma_{i_{1}i_{2}\cdots i_{N}}$, where $i_{j} \in \{0,1\}$, of the following $2^{N}$ surjective homomorphisms
$$\theta_{i_{1}i_{2}\cdots i_{N}}:\Gamma \to {\mathbb Z}_{2}=\langle x: x^{2}=1\rangle$$
defined by
$$\theta_{i_{1}i_{2}\cdots i_{N}}(\alpha_{j})=x^{i_{j}}, \quad j=1,\ldots,N,$$
$$\theta_{i_{1}i_{2}\cdots i_{N}}(\delta_{j})=x.$$

In this way, we obtain $2^{N}$ hyperelliptic Riemann surfaces 
$$R_{i_{1}i_{2}\cdots i_{N}}= {\mathbb H}^{2}/ \Gamma_{i_{1}i_{2}\cdots i_{N}}$$
admitting a hyperelliptic involution $\tau_{i_{1}i_{2}\cdots i_{N}}$ (induced by $\Gamma$) and a degree two branched covering (with deck group $\langle \tau_{i_{1}i_{2}\cdots i_{N}} \rangle$)
$$\pi_{i_{1}i_{2}\cdots i_{N}}:R_{i_{1}i_{2}\cdots i_{N}} \to \widehat{\mathbb C} \setminus \{q_{1},\ldots,q_{N+1}\},$$
with the same branch locus ${\mathcal B}$.

\begin{rema}[If $N \geq 1$, then ${\mathcal B}$ does not determine the pair $(R,\tau)$]
If $N=0$, then there is a unique index two torsion free subgroup of $\Gamma$. But, if $N \geq 1$, then we have $2^{N}$ possibilities. In this situation, 
let $R_{1}={\mathbb H}^{2}/\Gamma_{1}$ and $R_{2}={\mathbb H}^{2}/\Gamma_{2}$ (where $\Gamma_{j}$ is an index two torsion free subgroup of $\Gamma$) be any two of these hyperelliptic Riemann surfaces and denote by $\tau_{j}$ (respectively, $\pi_{j}$) the corresponding hyperelliptic involution on $R_{j}$ (respectively, the above degree two branched covering). If there is a biholomorphism $\varphi:R_{1} \to R_{2}$ that conjugates $\tau_{1}$ to $\tau_{2}$, then it induces a M\"obius transformation $M$ such that $M \circ \pi_{2} = \pi_{1} \circ \varphi$. In particular, $M$ keeps invariant the sets $\{q_{1},\ldots,q_{N+1}\}$ and also the set ${\mathcal B}$. So, for instance, if $N \geq 3$, then the first condition asserts that, for generic choices of the points $q_{j}$, $M=I$. But this condition asserts that $\varphi$ must be induced by an element of $\Gamma$ that must conjugates $\Gamma_{1}$ onto $\Gamma_{2}$. As $\Gamma_{1}$ is a normal subgroup of $\Gamma$, it means that $\Gamma_{1}=\Gamma_{2}$. In this way, we obtain non-isomorphic pairs $(R_{1},\tau_{1})$ and $(R_{2},\tau_{2})$ for which $R_{1}/\langle \tau_{1}\rangle=R_{2}/\langle \tau_{2}\rangle$, which is not possible in the closed situation.
\end{rema}

\begin{question}
Is the hyperelliptic involution unique?
\end{question}

%%%%%%%%%%%%%
\subsubsection{\bf Algebraic descriptions}\label{Sec:ecuaciones}
If 
$$T(z)=\left\{\begin{array}{ll}
\dfrac{(q_{3}-q_{1})(z-q_{2})}{(q_{3}-q_{2})(z-q_{1})},& N\geq 2;\\
\\
\dfrac{z-q_{2}}{z-q_{1}},& N=1;\\
\\
\dfrac{q_{1}}{q_{1}-z},& N=0;\\
\end{array}
\right.
$$
then 
$$T \circ \pi_{i_{1}i_{2}\cdots i_{N}}:R_{i_{1}i_{2}\cdots i_{N}} \to \widehat{\mathbb C} \setminus \{\infty,0,1,T(q_{4}),\ldots,T(Q_{N+1})\}$$
is a degree two branched covering, with deck group $\langle \tau_{i_{1}i_{2}\cdots i_{N}} \rangle$, and whose branch value set is, respectively,
$$\{T(\infty), T(0), T(1), T(\lambda_{1}), \ldots\} \subset{\mathbb C} \setminus \{0,1,T(q_{4}),\ldots,T(Q_{N+1})\}, \; N\geq 3.$$
$$\{T(\infty), T(0), T(1), T(\lambda_{1}), \ldots\} \subset{\mathbb C} \setminus \{0,1\}, \; N=2.$$
$$\{T(\infty), T(0), T(1), T(\lambda_{1}), \ldots\} \subset{\mathbb C} \setminus \{0\}, \; N=1.$$
$$\{T(\infty), T(0), T(1), T(\lambda_{1}), \ldots\} \subset{\mathbb C}, \; N=0.$$

The Weierstrass product theorem asserts that given any collection of different points $p_{k} \in {\mathbb C}$, $k=1,2,\ldots$, such that $\lim_{k \to +\infty}p_{k} =\infty$, there is an entire map $f:{\mathbb C} \to {\mathbb C}$ whose only zeroes are $p_{k}$ (all of them simple) and that such map is unique up to multiplication by $e^{m(z)}$, for some entire map $m(z)$. Moreover, if there is some integer $M>0$ such that (assuming $p_{k} \neq 0$ for $k \geq 2$)
$$\sum_{k=2}^{ \infty}1/|p_{k}|^{M+1}<\infty,$$
then such a map $f$ can be chosen of the form

$$f(z)=\left\{\begin{array}{ll}
\prod_{k=1}^{\infty} \left(1-\frac{z}{p_{k}}\right) e^{\frac{z}{p_{k}}+\frac{z^{2}}{2p_{k}^{2}}+\cdots+\frac{z^{M}}{M p_{k}^{M}}}, \; \mbox{if $p_{k} \neq 0$ for every  $k$}.\\
z\prod_{k=2}^{\infty} \left(1-\frac{z}{p_{k}}\right) e^{\frac{z}{p_{k}}+\frac{z^{2}}{2p_{k}^{2}}+\cdots+\frac{z^{M}}{M p_{k}^{M}}}, \; \mbox{if $p_{1} =0$}.
\end{array}
\right.
$$

Now, for each $j=1,\ldots,N+1$,
set $A_{j}=T(B_{j})$  and 
let us consider an entire map $f_{j}:{\mathbb C} \to {\mathbb C}$ such that:
\begin{enumerate}
\item the set of zeroes of $f_{1}$ is $A_{1}$, each one being simple;
\item the set of zeroes of $f_{2}$ is $\{1/p:p \in A_{2}\}$, each one being simple;
\item the set of zeroes of $f_{3}$ is $\{1/(p-1):p \in A_{3}\}$, each one being simple;
\item for $j \geq 4$, the set of zeroes of $f_{j}$ is $\{1/(p-T(q_{j})):p \in A_{j}\}$, each one being simple.
\end{enumerate}

Let $h:{\mathbb C} \setminus\{0,1,T(q_{1}),\ldots,T(q_{N+1})\} \to {\mathbb C}$ be a holomorphic map without zeroes. Then, the above $2^{N}$ hyperelliptic Riemann surfaces can be described by
$$w^{2}=h(z)f_{1}(z)f_{2}\left(\frac{1}{z}\right)f_{3}\left(\frac{1}{z-1}\right) \prod_{j=4}^{N+1} f_{j}\left(\frac{1}{z-T(q_{j})}\right).$$

Note that we may forget the exponential factors $e^{m_{j}(z)}$ in the forms of the $F_{j}(z)$.

In the above algebraic model, as was for the compact case, the hyperelliptic involution $\tau$ is given by $\tau(z,w)=(z,-w)$.

%%%%%%%%%%%%%
\subsection{Hyperelliptic Riemann surfaces on the LNM}
As previously noted, the only case we obtain a hyperelliptic Riemann surface on the LNM is for $N=0$.
In this case, ${\mathcal B}=B_{1}=\{\infty,0,1,\lambda_{1},\lambda_{2},\ldots\}$, with ${\mathcal B}'=\{q_{1}\}$, 
the generalized Fermat pair is $(C^{2}_{\infty}({\mathcal B};\emptyset),H_{\infty}^{2})$ and the corresponding Galois branched covering $\pi_{\infty}:C^{2}_{\infty}({\mathcal B};\emptyset) \to \widehat{\mathbb C} \setminus \{q_{1}\}$.
If  $T(z)=q_{1}/(q_{1}-z)$, then 
$P=T \circ \pi_{\infty}:C^{2}_{\infty}({\mathcal B};\emptyset) \to {\mathbb C}$ is a Galois branched covering with $H^{2}_{\infty}$ as its deck group. The branch values set of $P$ is given by 
$${\mathcal B}_{P}:=\{\mu_{1}=0, \mu_{2}=1, \mu_{3}=q_{1}/(q_{1}-1), \mu_{4}=q_{1}/(q_{1}-\lambda_{1}), \ldots\}$$ which accumulates at $\infty$.

The hyperelliptic Riemann surface is $R=C^{2}_{\infty}({\mathcal B};\emptyset)/F$, where $F$ is the unique index two subgroup in $H_{\infty}^{2}$ acting freely on $C^{2}_{\infty}({\mathcal B};\emptyset)$
(i.e., which does not contain the generators $a_{j}$), which is the kernel of 
the surjective homomorphism
$$\theta:H_{\infty}^{2} \to {\mathbb Z}_{2}=\langle a : a^{2}=1\rangle: 
a_{j} \mapsto \theta(a_{j})=a,
$$
that is,
$F=\langle a_{j}a_{1}: j \geq 2 \rangle.$

The cyclic group $L=H/F \cong {\mathbb Z}_{2}$ is the group generated by a hyperelliptic involution $\tau$.  The surface
$R$ is described by the affine infinite plane curve
$$C:\left\{\begin{array}{c}
w^{2}=zf(z)
\end{array}
\right\} \subset {\mathbb C}^{2},
$$
where $f:{\mathbb C} \to {\mathbb C}$ is an entire map whose zeros (all of them simple ones) are given by the points $\mu_{j}$, for $j \geq 2$.
In this model,
$\tau(z,w)=(z,-w)$.

In the case that  there is some integer $M>0$ such that 
$$\sum_{k=2}^{ \infty}1/|\mu_{k}|^{M+1}<\infty,$$
then

$$C:\left\{\begin{array}{c}
w^{2}=z\prod_{k=2}^{\infty} \left(1-\frac{z}{\mu_{k}}\right) e^{\frac{z}{\mu_{k}}+\frac{z^{2}}{2\mu_{k}^{2}}+\cdots+\frac{z^{M}}{M \mu_{k}^{M}}}
\end{array}
\right\} \subset {\mathbb C}^{2},
$$
and
$\tau(z,w)=(z,-w)$.
The cyclic group $G=\langle a \rangle \cong {\mathbb Z}_{2}$ is the deck group of the degree two Galois branched covering
$$\pi:C \to {\mathbb C}:(z,w) \mapsto z,$$
whose branch values is the set ${\mathcal B}_{P}$.

%%%%%%%%%%%%%% 
\subsubsection{\bf A remark}
Each time we have a surjective homomorphism $\theta:H^{2}_{\infty} \to {\mathbb Z}_{2}^{m}$, whose kernel $F$ does not contains any of the elements $a_{1}, a_{2}, \ldots$, then $R=C_{\infty}^{2}({\mathcal B};\emptyset)/F$ is a Riemann surface of infinite type admitting the group $J={\mathbb Z}_{2}^{m}$ as a group of conformal automorphisms such that $R/J=C_{\infty}^{2}/H_{\infty}^{2}$. A natural question is if we may find algebraic equations for $R$.
Below, we consider $m=2$. Let us consider a decomposition ${\mathbb N}=A_{1} \cup A_{2} \cup A_{3}$, where $A_{i} \cap A_{j}=\emptyset$ and the surjective homomorphism
$$\theta:H_{\infty}^{2} \to {\mathbb Z}_{2}^{2}=\langle a, b: a^{2}=b^{2}=(ab)^{2}=1\rangle$$
$$\theta(a_{j}) \in \left\{ \begin{array}{ll}
a, & j \in A_{1}\\
b, & j \in A_{2}\\
ab, & j \in A_{3}
\end{array}
\right\}.
$$

The kernel of $\theta$ is
$$F=\langle a_{i}a_{j}: (i,j) \in A_{1}^{2} \cup A_{2}^{2} \cup A_{3}^{2}\rangle,$$
and it acts freely on $C^{2}_{\infty}({\mathcal B};\emptyset)$.
The Riemann surface $R=C^{2}_{\infty}({\mathcal B};\emptyset)/F$ admits the group $L=H/F \cong {\mathbb Z}_{2}^{2}$ as group of conformal automorphisms such that $R/L=C^{2}_{\infty}({\mathcal B};\emptyset)/H$. 
Now, for each $A_{j}$, we  may find a holomorphic map $f_{A_{j}}:{\mathbb C} \to {\mathbb C}$  with simple zeroes exactly at the points in $\mu_{j}$ for $j \in A_{j}$. Then
we may consider the affine infinite curve
$$C:\left\{\begin{array}{c}
w^{2}=f_{A_{1}}(z) f_{A_{3}}(z)\\
u^{2}=f_{A_{2}}(z) f_{A_{3}}(z)
\end{array}
\right\} \subset {\mathbb C}^{3},
$$
which admits the automorphisms 
$$a(z,u,w)=(z,-u,w), \; b(z,u,w)=(z,u,-w).$$

The group $G=\langle a,b \rangle \cong {\mathbb Z}_{2}^{2}$ is the deck group of the Galois branched covering
 $$\pi:C \to {\mathbb C}:(z,u,w) \mapsto z,$$
whose branch set is ${\mathcal B}_{P}$. Moreover, (a) the set of points $\mu_{j}$, for $j \in A_{1}$, is the image of the fixed points of $a$, 
(b) the set of points $\mu_{j}$, for $j \in A_{2}$ the projection of the fixed points of $b$ and
(c) the set of points $\mu_{j}$, for $j \in A_{3}$ the projection of the fixed points of $ab$.
Again, by the maximality property (in the definition of generalized Fermat curves of infinite type), we must have a subgroup $J<H_{\infty}^{2}$ of index four (and necessarily acting freely on $C_{\infty}^{2}({\mathcal B};\emptyset)$) such that $C=C_{\infty}^{2}({\mathcal B};\emptyset)/J$. By the above properties on $A_{1}$, $A_{2}$ and $A_{3}$, we may see that 
 $J=F$. This in particular permits us to see that $C=R$ and that $G=L$.

%%%%%%%%%%%%%%%%%%
%%%%%%%%%%%%%%%%%%

\end{document}